\documentclass[reqno,12pt]{amsart}
\pdfoutput=1
\usepackage{amsmath,amsthm,amsfonts,amssymb,ifpdf}
\usepackage{amssymb}
\usepackage{amsmath}
\usepackage{amsthm}
\usepackage{graphicx}
\usepackage[all]{xy}
\usepackage{enumerate}
\usepackage{tikz-cd}

\newcommand{\RomanNumeralCaps}[1]
    {\MakeUppercase{\romannumeral #1}}
\theoremstyle{plain} 
\newtheorem{theorem}{Theorem}[section]

\newtheorem{lemma}[theorem]{Lemma}

\theoremstyle{definition} \newtheorem{definition}[theorem]{Definition}

\theoremstyle{remark} \newtheorem{remark}[theorem]{Remark}

\usepackage[all]{xy}

\ifpdf
  \usepackage[
    pdftex,
    colorlinks,%
    linkcolor=blue,citecolor=red,urlcolor=blue,
    hyperindex,%
    plainpages=false,%
    bookmarksopen,%
    bookmarksnumbered%
  ]{hyperref} 
 
 \usepackage{thumbpdf}
\else
  \usepackage{hyperref}
\fi

\newcommand\nnfootnote[1]{%
  \begin{NoHyper}
  \renewcommand\thefootnote{}\footnote{#1}%
  \addtocounter{footnote}{-1}%
  \end{NoHyper}
}
\def\deg{{\rm deg}}
\def\Sym{{\rm Sym}}
\def\Hilb{{\rm Hilb}}
\def\Quot{{\rm Quot}}
\def\Ker{{\rm Ker}}
\def\Pic{{\rm Pic}}
\def\Div{{\rm Div}}

\newcommand{\ncom}{\newcommand}
\ncom{\mylabel}[1]{{\rm (#1)}\label{#1}}
\ncom{\Hom}{{\textit{Hom}}}
\ncom{\eop}{{\hfill $\Box$}}
\begin{document}
\baselineskip=16pt

\nnfootnote{Mathematics Classification Number: 14C05, 14C20, 14C22, 14L30.}
\nnfootnote{Keywords:  Diagonal, ind-variety, divisor, Hilbert scheme, good partition.}

\setcounter{tocdepth}{1}

\title[Diagonal property and weak point property]{Diagonal property and weak point property of higher rank divisors and certain Hilbert schemes}

\author{Arijit Mukherjee}
\address{Department of Mathematics, Indian Institute of Science Education and Research Tirupati, Karakambadi Rd,  Opp Sree Rama Engineering College, Rami Reddy Nagar, Mangalam, Tirupati, Andhra Pradesh - 517 507, India.}
\email{mukherjee90.arijit@gmail.com}

\begin{abstract}
In this paper, we introduce the notion of the diagonal property and the weak point property for an ind-variety.  We prove that the ind-varieties of higher rank divisors of integral slopes on a smooth projective curve have the weak point property.  Moreover, we show that the ind-variety of $(1,n)$-divisors has the diagonal property.  Furthermore, we obtain that the Hilbert schemes associated to the good partitions of a constant polynomial satisfy the diagonal property.  On the process of obtaining this, we provide an upper bound on the number of such Hilbert schemes upto isomorphism.  Furthermore, we prove that the obtained upper bound is attained in case of genus zero curves and hence conclude that the bound is sharp.   
\end{abstract}
\maketitle

\tableofcontents

\section{Introduction}
\label{sec:Introduction}
Let $X$ be a smooth projective variety over the field of complex numbers.  By the diagonal subscheme of $X$, denoted by $\Delta_X$, one means the image of the embedding $\delta: X\rightarrow X\times X$ given by $\delta(x)=(x,x)$, where $x\in X$.  This subscheme plays a central role in intersection theory.  In fact, to get hold of the fundamental classes of any subschemes of a variety $X$, it's enough to know the fundamental class of the diagonal $\Delta_X$ of $X$, (cf. \cite{P}).  
 
In this paper, we talk about the diagonal property and the weak point property of some varieties.  Broadly speaking, the diagonal property of a variety $X$ is a property which demands a special structure of the diagonal $\Delta_X$ and therefore very significant to study from the viewpoint of intersection theory.  Moreover, being directly related to the diagonal subscheme $\Delta_X$, this property imposes strong conditions on the variety $X$ itself.  For example, this property is responsible for the existence or non-existence of cohomologically trivial line bundles on $X$.  The weak point property is also very much similar to diagonal property but a much weaker one.  Both of these notions were introduced in \cite{PSP}.  Many mathematicians have studied about the diagonal property and the weak point property of varieties, (cf. \cite{F}, \cite{FP}, \cite{LS}).  In this paper, we introduce these two notions for an ind-variety, that is an inductive system of varieties and showed that the ind-varieties of higher rank divisors of integral slope on a smooth projective curve $C$  satisfy these properties.  Also, we show that some Hilbert schemes associated to good partitions of a polynomial satisfy the diagonal property.

Before mentioning the results obtained in this paper more specifically, let us fix some notations which we are going to use repeatedly.  We denote by $\mathbb{C}$ the field of complex numbers.  In this paper, by $C$ we always mean a smooth projective curve over $\mathbb{C}$. The notation $\mathcal{O}_C$ is reserved for the structure sheaf over $C$.  For a given divisor $D$ on $C$, by $\mathcal{O}_C(D)$ we mean the corresponding line bundle over $C$ and denote its degree by $\deg(D)$.  By $\Sym^d(C)$ we denote the $d$-th symmetric power of the curve $C$. For a given positive integer $n$ and a locally free sheaf (equivalently, a vector bundle) $\mathcal{F}$ over $C$, by $\mathcal{F}^n$ we mean the direct sum of $n$ many copies of $\mathcal{F}$.  By $\Quot^{d}_{\mathcal{G}}$ we denote the Quot scheme parametrizing all torsion quotients of $\mathcal{G}$ having degree $d$, $\mathcal{G}$ being any coherent sheaf on $C$.  For a given polynomial $P(t)\in \mathbb{Q}[t]$, we denote the Quot scheme parametrizing all torsion quotients of $\mathcal{G}$ having Hilbert polynomial $P(t)$ by $\Quot^{P}_{\mathcal{G}}$.
 
Let us now go through the chronology of this paper in a bit more detail.  The manuscript is arranged as follows.  In Section \ref{sec: dp and wpp}, we recall the definitions of the  diagonal property and the weak point property for a smooth projective variety and talk about a relation between these two properties.  Moreover, for a smooth projective curve $C$ over $\mathbb{C}$, we recall a couple of relevant results about the variety $\Sym^d(C)$ and the Quot scheme $\Quot^{d}_{\mathcal{O}_C^n}$.  In Section \ref{sec: Higher rank divisors}, we recall the definition of $(r,n)$-divisors on $C$ \& the ind-variety made out of such divisors.  We then precisely define, what we mean by the diagonal property and the weak point property of an ind-variety and prove the following theorems followed by that.
\begin{theorem}
Let $C$ be a smooth projective curve over $\mathbb{C}$.  Also let $r\geq 1$ and $n$ be two integers.  Then the ind-variety of $(r,n)$-divisors having integral slope on $C$ has the weak point property. 
\end{theorem} 
\begin{theorem}
Let $C$ be a smooth projective curve over $\mathbb{C}$ and $n$ any given integer. Then the ind-variety of $(1,n)$-divisors on $C$ has the diagonal property.
\end{theorem}           
In Section \ref{sec: Hilbert scheme}, we deal with the Hilbert scheme associated to a polynomial $P$ and its good partition.  E.~Bifet has dealt with these schemes in \cite{B}.  Moreover, he showed that the Quot scheme $\Quot_{\mathcal{O}_C^r}^P$ can be written as disjoint union of some smooth, the torus $\mathbb{G}_m^r$-invariant, locally closed vector bundles over the mentioned Hilbert schemes.  Here, we talk about the diagonal property of such Hilbert schemes.  Towards that, we prove the following lemma.
\begin{lemma}
Let $n$ be a given positive integer.  Then any partition of $n$ is also a good partition of $n$ and vice versa.
\end{lemma}  
In the lemma stated above, we interpret the integer $n$ as a constant polynomial and therefore it makes sense to talk about good partition of $n$.  Using this, we finally prove the following theorem, which not only provides some Hilbert scheme associated to good partitions of a polynomial but also gives a sharp upper number of such Hilbert schemes. Precisely, we obtain:
\begin{theorem}
Let $C$ be a smooth projective curve over $\mathbb{C}$ and $n$ a positive integer.  Then there are at most as many Hilbert schemes $\Hilb^{\underline{n}}_C$ (upto isomorphism) associated to the constant polynomial $n$ and its good partitions $\underline{n}$ satisfying diagonal property as there are conjugacy classes of the symmetric group $S_n$ of $n$ symbols.  Moreover, this upper bound is achieved when $C$ is of genus $0$ and hence is sharp.
\end{theorem} 

Here to prove the sharpness of the obtained upper bound, we show that for genus $0$ curves the upper bound is attained, mainly using the following lemma.
\begin{lemma}  
Let $n$ be a positive integer.  Let $(m_1,m_2,\ldots,m_s)$ and $(n_1,n_2,\ldots,n_s)$ be two distinct partitions of $n$ of same length $s$.  Then $\mathbb{P}^{m_1}\times 
\mathbb{P}^{m_2}\times \cdots \times \mathbb{P}^{m_s}$ is not isomorphic to  $\mathbb{P}^{n_1}\times \mathbb{P}^{n_2}\times \cdots \times \mathbb{P}^{n_s}$.
\end{lemma}

\section{On the diagonal property and the weak point property of a variety}\label{sec: dp and wpp} 
In this section, we recall the notions of the diagonal property and the weak point property of a variety and talk about relations between these two properties.  Moreover, for a smooth projective curve $C$ over $\mathbb{C}$, we recall a couple of relevant results about the variety $\Sym^d(C)$ and the Quot scheme $\Quot^{d}_{\mathcal{O}_C^n}$.   

Let us begin with the precise definitions of the diagonal property and the weak point property of a variety.
\begin{definition}\label{dp}
Let $X$ be a variety over the field of complex numbers.  Then $X$ is said to have the diagonal property if there exists a vector bundle $E\rightarrow X\times X$ of rank equal to the dimension of $X$, and a global section $s$ of $E$ such that the zero scheme $Z(s)$ of $s$ coincides with the diagonal $\Delta_X$ in $X\times X$.  
\end{definition} 
\begin{definition}\label{wpp}
Let $X$ be a variety over the field of complex numbers.  Then $X$ is said to have the weak point property if there exists a vector bundle $F\rightarrow X$ of rank equal to the dimension of $X$, and a global section $t$ of $F$ such that the zero scheme $Z(s)$ of $s$ is a reduced point of $X$.
\end{definition}
\begin{remark}\label{dp implies wpp}
It can be noted immediately that for a variety, having the weak point property is in fact a weaker condition than having the diagonal property.  To prove this precisely, let's stick to the notations of Definition \ref{dp} and \ref{wpp}.  Let us choose a point $x_0\in X$.  Then $Z(s|_{X\times \{x_0\}})=\{x_0\}$.  Therefore, the diagonal property implies the weak point property. 
\end{remark}
We now quickly go through some results related to the diagonal property and the weak point property of two varieties which arise very naturally from a given curve $C$.  To be specific, we look upon the varieties $\Sym^d(C)$ and $\Quot^{d}_{\mathcal{O}_C^n}$.  We mention a couple of results in this context.  These are due to \cite{BS}.
\begin{theorem}\label{symm prod_dp}
Let $C$ be a smooth projective curve over $\mathbb{C}$.  Then, the $d$-th symmetric product $\Sym^d(C)$ of the curve $C$ has the diagonal property for any positive integer $d$.
\end{theorem}
\begin{proof}
See \cite[Theorem 3.1, p. 447]{BS}.
\end{proof}
\begin{theorem}\label{Quot scheme_wpp}
Let $C$ be a smooth projective curve over $\mathbb{C}$. Let $d$ and $n$ be two given positive integer such that $n|d$.  Then the Quot scheme $\Quot^{d}_{\mathcal{O}_C^n}$ parametrizing the torsion quotients of $\mathcal{O}_C^n$ of degree $d$ has the weak point property.    
\end{theorem}
\begin{proof}
See \cite[Theorem 2.2, p. 446-447]{BS}.
\end{proof}  
\begin{remark}\label{positivity required in the hypothesis}
Let us discuss about the hypothesis of Theorem \ref{Quot scheme_wpp}.  Firstly, positivity of the integer $n$ is necessary as we are talking about the sheaf $\mathcal{O}_C^n$.  Moreover, if we assume that $d$ is a positive integer and $n|d$, then there exists a positive integer $r$ such that $d=nr$.  The positivity of this integer $r$ is heavily used in the proof.  Indeed, the authors first showed that to prove Theorem \ref{Quot scheme_wpp}, it is enough to show that the Quot scheme $\Quot^{d}_{L^n}$ has the weak point property for some degree $r$ line bundle $L$ over $C$.  Now the line bundle $L$ is taken to be the line bundle $\mathcal{O}_C(rx_0)$, where $x_0\in X$.  Now positivity of $r$ gives the natural inclusion $i:\mathcal{O}_C \hookrightarrow \mathcal{O}_C(rx_0)$.  This in turn gives the following short exact sequence:
\begin{equation}\label{ses_wpp of Quot scheme}
0\rightarrow \mathcal{O}_C^n \rightarrow \mathcal{O}_C(rx_0)^n \rightarrow T \rightarrow 0.
\end{equation}   
Now the torsion sheaf $T$ as in \eqref{ses_wpp of Quot scheme} lies in the sheaf $\Quot^{d}_{{\mathcal{O}_C(rx_0)}^n}$, the sheaf they wanted to work on to prove the required result.  So, positivity of $d$ has a huge role to play in the proof. 
\end{remark}
\begin{remark}
It is worthwhile to note a connection between Theorem \ref{symm prod_dp} \& \ref{Quot scheme_wpp}.  If we take, $r=1$, then Theorem \ref{Quot scheme_wpp} says that for any positive integer $d$, the Quot scheme $\Quot^{d}_{\mathcal{O}_C^n}$ has the weak point property.  As, $\Sym^d(C)\cong \Quot^{d}_{\mathcal{O}_C}$, by Remark \ref{dp implies wpp}, Theorem \ref{Quot scheme_wpp} follows from Theorem \ref{symm prod_dp} for $r=1$ case.
\end{remark}
\section{Higher rank divisors on a curve, corresponding ind-varieties and the diagonal \& the weak point property}
\label{sec: Higher rank divisors}
In this section, we recall the definition of higher rank divisors on a curve, corresponding ind-varieties and quasi-isomorphism between them.  Then we introduce the notion of the diagonal property and the weak point property for an ind-variety in general and prove some results about the ind-varieties of higher rank divisor in particular.   
 
Let us denote by $K$ the field of rational functions on $C$, thought as a constant $\mathcal{O}_C$-module.
\begin{definition}
	A divisor of rank $r$ and degree $n$ over $C$ is a coherent sub $\mathcal{O}_C$-module of $K^{\oplus r}$ having rank $r$ and degree $n$.  This is denoted by $(r,n)$-divisor.
\end{definition}
\begin{remark}
Since we take $C$ to be smooth, these $(r,n)$-divisors coincide with the matrix divisors defined by A. Weil, (cf. \cite{W}).
\end{remark}
Let us denote the set of all $(r,n)$-divisors on $C$ by $\Div^{r,n}$.  Let $D$ be an effective divisor of degree $d$ over $C$.  Then corresponding to $D$, let us define the following subset of $\Div^{r,n}$, denoted by $\Div^{r,n}(D)$ as follows:
\begin{definition}
	$\Div^{r,n}(D):=\{E \in \Div^{r,n} \mid E \subseteq \mathcal{O}_{C}(D)^{\oplus r}\}$.
\end{definition}
Then clearly we have, $\Div^{r,n}=\bigcup_{D\geq 0}\Div^{r,n}(D)$.
Also, the elements of $\Div^{r,n}(D)$ can be identified with the rational points of the Quot scheme $\Quot_{\mathcal{O}_C(D)^r}^m$, where $m=r\cdot \deg (D)-n$.  Therefore taking $D=\mathcal{O}_C$, we can say that the elements of $\Div^{r,n}(\mathcal{O}_C)$ can be identified with the rational points of the Quot scheme $\Quot_{\mathcal{O}_C^r}^{-n}$.

Let us now recall what one means by a inductive system of varieties.
\begin{definition}
	An ind-variety $\mathbf{X}=\{X_\lambda,f_{\lambda \mu}\}_{\lambda,\mu \in \Lambda}$ is an inductive system of complex algebraic varieties $X_\lambda$ indexed by some filtered ordered set $\Lambda$.  That is to say, an ind-variety is a collection $\{X_\lambda\}_{\lambda \in \Lambda}$ of complex algebraic varieties, where $\Lambda$ is some filtered ordered set, along with the morphisms $f_{\lambda \mu}: X_{\lambda}\rightarrow X_{\mu}$ of varieties for every $\lambda \leq \mu$ such that the following diagrams commute for every $\lambda \leq \mu \leq \nu$.
\begin{equation*}
\xymatrix{
X_{\lambda} \ar[rd]_{f_{\lambda \nu}}\ar[r]^{f_{\lambda \mu}} & X_{\mu}
		\ar[d]^{f_{\mu \nu}}\\ 
		 & X_{\nu}}
\end{equation*}	
\end{definition}
Taking the indexing set $\Lambda$ to be the set of effective divisors on $C$, we have the inclusion 
\begin{equation}\label{E6}
\Div^{r,n}(D_\alpha)\rightarrow \Div^{r,n}(D_\beta),
\end{equation}
induced by the closed immersion $\mathcal{O}_{C}(D_{\alpha})^{\oplus r}\hookrightarrow \mathcal{O}_{C}(D_{\beta})^{\oplus r}$ for any pair of effective divisors $D_\alpha,D_\beta$ satisfying $D_\alpha\leq D_\beta$.
\begin{definition}\label{Div_ind-variety}
	The ind-variety determined by the inductive system consisting of the varieties $\Div^{r,n}(D)$ and the closed immersions as in (\ref{E6}) is denoted by ${\mathbf{Div}}^{r,n}$.
\end{definition}

Now we are going to consider another ind-variety.  Given any effective divisor $D$ on $C$, we consider a complex algebraic variety $Q^{r,n}(D)$ defined as follows.
\begin{definition}\label{Quot schemes as constituent of ind-variety}
	$Q^{r,n}(D):=\Quot^{n+r\cdot \deg (D)}_{\mathcal{O}_C^r}$.
\end{definition}
Let $D_1$ and $D_2$ be any two effective divisors  with $D_2\geq D_1$.   Denoting $D_2-D_1$ as $D$, we have the following structure map denoted by $\mathcal{O}_C(-D)$.
\begin{equation*}
\mathcal{O}_C(-D): \Quot^{n+r\cdot \deg (D_1)}_{\mathcal{O}_C^r}\rightarrow \Quot^{n+r\cdot \deg (D_2)}_{\mathcal{O}_C^r},
\end{equation*}
where the map $\mathcal{O}_C(-D)$ means tensoring the submodules with $\mathcal{O}_C(-D)$.
Elaborately, let $(\mathcal{F},q) \in \Quot^{n+r\cdot \deg (D_1)}_{\mathcal{O}_C^r}$.  Therefore we have the following exact sequence:
\begin{equation*}
\xymatrix{ 0 \ar[r]& \Ker (q) \ar[r] &\mathcal{O}_C^r\ar[r]^{q} &\mathcal{F}
	\ar[r] &0},
\end{equation*}
where degree of $\mathcal{F}$ is $n+r\cdot \deg (D_1)$ and hence degree of $\Ker (q)$ is $-n-r\cdot \deg (D_1)$.  Tensoring this by $\mathcal{O}_C(-D)$ we get,
\begin{equation*}
\xymatrix{ 0 \ar[r]& \Ker (q)\otimes \mathcal{O}_C(-D) \ar[r] &\mathcal{O}_C(-D)^r\ar[r] &\mathcal{F}\otimes \mathcal{O}_C(-D)
	\ar[r] &0}.
\end{equation*}
Here $\deg(\Ker (q)\otimes \mathcal{O}_C(-D))=r\cdot(\deg (D_1)-\deg (D_2))-n-r\cdot \deg (D_1)=-n-r\cdot \deg (D_2)$.  Now as $\mathcal{O}_C(-D)^r$ sits inside $\mathcal{O}_C^r$, $\Ker (q)\otimes \mathcal{O}_C(-D)$ also sits inside $\mathcal{O}_C^r$.  Therefore we now get the following exact sequence:
\begin{equation*}
\xymatrix{ 0 \ar[r]& \Ker (q)\otimes \mathcal{O}_C(-D) \ar[r] &\mathcal{O}_C^r\ar[r]^{q_1} &\mathcal{F}_1
	\ar[r] &0},
\end{equation*} 
where $\deg(\mathcal{F}_1)=n+r\cdot \deg (D_2)$. Hence, $\mathcal{F}_1 \in \Quot^{n+r\cdot \deg (D_2)}_{\mathcal{O}_C^r}$.  Thus, the map $\mathcal{O}_C(-D): \Quot^{n+r\cdot \deg (D_1)}_{\mathcal{O}_C^r} \rightarrow \Quot^{n+r\cdot \deg (D_2)}_{\mathcal{O}_C^r}$ given by $(\mathcal{F},q) \mapsto (\mathcal{F}_1,q_1)$ is well defined.  Therefore for $D_2\geq D_1$ we have,
\begin{equation}\label{E7}
\mathcal{O}_C(-D):Q^{r,n}(D_1)\rightarrow Q^{r,n}(D_2).
\end{equation}
\begin{definition}\label{Quot_ind-variety}
	The ind-variety determined by the inductive system consisting of the varieties $Q^{r,n}(D)$ and the morphisms as in (\ref{E7}) is denoted by ${\mathbf{Q}}^{r,n}$.
\end{definition}

Let us clarify what we mean by a good enough morphism in the category of ind-varieties.   
\begin{definition}
	Let $\mathbf{X}=\{X_D,f_{DD_1}\}_{D,D_1\in \mathcal{D}}$ and $\mathbf{Y}=\{Y_D,g_{DD_1}\}_{D,D_1\in \mathcal{D}}$ be two inductive system of complex algebraic varieties, where $\mathcal{D}$ is the ordered set of all effective divisors on $C$.  Then by a morphism $\mathbf{\Phi}=\{\alpha,\{\phi_D\}_{D\in \mathcal{D}}\}$ from $\mathbf{X}$ to $\mathbf{Y}$ we mean an order preserving map $\alpha:\mathcal{D}\rightarrow \mathcal{D}$ together with a family of morphisms $ \phi_D: X_D \rightarrow Y_{\alpha(D)}$ satisfying the following commutative diagrams for all $D,D_1\in \mathcal{D}$ with $D\leq D_1$.
	\begin{equation*}
	\label{eq:277}
	\xymatrix{X_D \ar[d]_{f_{DD_1}}\ar[rr]^{\phi_D} && Y_{\alpha(D)}
		\ar[d]^{g_{\alpha(D)\alpha(D_1)}}\\ 
		X_{D_1}\ar[rr]^{\phi_{D_1}}
		&& Y_{\alpha(D_1)}
	}
	\end{equation*}	
\end{definition}
\begin{remark}
Note that $\alpha:\mathcal{D}\rightarrow \mathcal{D}$ being an order preserving map, $D\leq D_1\Rightarrow \alpha(D)\leq \alpha(D_1)$.  Therefore the map $g_{\alpha(D)\alpha(D_1)}:Y_{\alpha(D)}\rightarrow Y_{\alpha(D_1)}$ makes sense.
\end{remark}
\begin{definition}
	Let $\mathbf{X}=\{X_D,f_{DD_1}\}_{D,D_1\in \mathcal{D}}$ and $\mathbf{Y}=\{Y_D,g_{DD_1}\}_{D,D_1\in \mathcal{D}}$ be two inductive system of complex algebraic varieties.  Then a morphism $\mathbf{\Phi}=\{\alpha,\{\phi_D\}_{D\in \mathcal{D}}\}$ from $\mathbf{X}$ to $\mathbf{Y}$ is said to be a quasi-isomorphism if
	\begin{enumerate}[(a)]
		\item $\alpha(\mathcal{D})$ is a cofinal subset of $\mathcal{D}$,
		\item given any integer $n$ there exists $D_n\in \mathcal{D}$ such that for all $D\geq D_n$, $ \phi_D: X_D \rightarrow Y_{\alpha(D)}$ is an open immersion and codimension of $Y_{\alpha(D)}-\phi_D(X_D)$ in $Y_{\alpha(D)}$ is greater than $n$, i.e for $D\gg 0$ the maps $\phi_D: X_D \rightarrow Y_{\alpha(D)}$ are open immersion and very close to being surjective.
	\end{enumerate}
\end{definition}

Now we recall an important theorem which talks about the quasi-isomorphism between the ind-varieties defined in Definition \ref{Div_ind-variety} and \ref{Quot_ind-variety}.

\begin{theorem}\label{T2}
	There is a natural quasi-isomorphism between the ind-varieties ${\mathbf{Div}}^{r,n}$ and ${\mathbf{Q}}^{r,-n}$.
\end{theorem}
\begin{proof}
	See \cite[Remark, page-647]{BGL}.  Infact, let $D$ be an effective divisor on $C$ of degree $d$. Let $(\mathcal{F},q)\in \Quot_{\mathcal{O}_C(D)^r}^{rd-n}$.  Then we have the following exact sequence.
	\begin{equation*}
	\xymatrix{ 0 \ar[r]& \Ker (q) \ar[r] &\mathcal{O}_C(D)^r\ar[r]^{q} &\mathcal{F}
		\ar[r] &0},
	\end{equation*}
	where $\deg(\mathcal{F})=rd-n$. Tensoring this with $\mathcal{O}_C(-D)$ we get,
	\begin{equation*}
	\xymatrix{ 0 \ar[r]& \Ker (q)\otimes \mathcal{O}_C(-D) \ar[r] &\mathcal{O}_C^r\ar[rr]^{q_1} &&\mathcal{F}\otimes \mathcal{O}_C(-D)
		\ar[r] &0},
	\end{equation*}
	where $\deg(\mathcal{F}\otimes \mathcal{O}_C(-D))=rd-n$.  Hence, $(\mathcal{F}\otimes \mathcal{O}_C(-D),q_1)\in \Quot_{\mathcal{O}_C^r}^{rd-n}.$  So we get a map $\Quot_{\mathcal{O}_C(D)^r}^{rd-n}\rightarrow \Quot_{\mathcal{O}_C^r}^{rd-n}$.  Restricting this map to the rational points of $\Quot_{\mathcal{O}_C(D)^r}^{rd-n}$, we obtain a map $\Div^{r,n}(D)\rightarrow Q^{r,-n}(D)$.  This map in turn will induce the required quasi-isomorphism 
	\begin{equation*}
	{\mathbf{Div}}^{r,n}\rightarrow {\mathbf{Q}}^{r,-n}.
	\end{equation*}  
\end{proof}
\begin{remark}\label{ind-variety of divisors}
By Theorem \ref{T2}, we can interpret ${\mathbf{Q}}^{r,-n}.$ as the ind-variety of $(r,n)$-divisors on $C$.
\end{remark}
Now we are in a stage to describe what we mean by the diagonal property and the weak point property of an ind-variety.  In this regard, we have couple of definitions as follows.  The notion of smoothness of an ind-variety (cf. \cite[\S 2, p. 643]{BGL}) motivates us to define the following two notions relevant to our context. 
\begin{definition}\label{indvariety_dp}
Let $\Lambda$ be a filtered ordered set.  Let $X=\{X_{\lambda}, f_{\lambda \mu}\}_{\lambda, \mu \in \Lambda}$ be an ind-variety.  Then $X$ is said to have the diagonal property (respectively weak point property) if there exists some $\lambda_0\in \Lambda$ such that for all $\lambda\geq \lambda_0$, the varieties $X_{\lambda}$'s have the diagonal property (respectively weak point property).
\end{definition}

Let us now associate a rational number to a given higher rank divisor.  In fact, this number helps us to find some ind-varieties having the diagonal property and weak point property. 
\begin{definition}\label{slope of higher rank divisor}
For a given $(r,n)$-divisor, the rational number $\tfrac{n}{r}$ is said to its slope.
\end{definition}
 
We now prove a couple of theorems about the diagonal property and weak point property of ind-varieties of $(r,n)$-divisors, when the rational number as in Definition \ref{slope of higher rank divisor} is in fact an integer.
\begin{theorem}\label{integral slope case}
Let $C$ be a smooth projective curve over $\mathbb{C}$.  Also let $r\geq 1$ and $n$ be two integers.  Then the ind-variety of $(r,n)$-divisors having integral slope on $C$ has the weak point property. 
\end{theorem}
\begin{proof}
It can be noted that a $(r,n)$-divisor is of integral slope if and only if $n$ is an integral multiple of $r$, by Definition \ref{slope of higher rank divisor}.  Therefore, the ind-variety ${\mathbf{Div}}^{r,kr}$, or equivalently ${\mathbf{Q}}^{r,-kr}$ by Remark \ref{ind-variety of divisors}, is the ind-variety of higher rank divisors of integral slope.

Let $D$ be an effective divisor of degree $d$ on $C$.  Then we have, $Q^{r,-n}(D)=\Quot_{\mathcal{O}_C^r}^{rd-n}$ by Definition \ref{Quot schemes as constituent of ind-variety}.  Now if $n=rk$ for some integer $k$, then $Q^{r,-rk}(D)=\Quot_{\mathcal{O}_C^r}^{rd-rk}=\Quot_{\mathcal{O}_C^r}^{r(d-k)}$.  Now let's pick an effective divisor $D_0$ of degree $d_0$ satisfying the inequality $d_0>k$.  Then for all $D\geq D_0$ and $n=rk$, we have 
\begin{equation}\label{eqn_1_wpp}
\deg(D)\geq \deg(D_0)=d_0>k, 
\end{equation}
and 
\begin{equation}\label{eqn_2_wpp}
Q^{r,-n}(D)=Q^{r,-rk}(D)=\Quot_{\mathcal{O}_C^r}^{r(\deg(D)-k)}.
\end{equation}
Here $r(\deg(D)-k)$ is a positive integer by \eqref{eqn_1_wpp}.  Therefore, by Theorem \ref{Quot scheme_wpp} and Definition \ref{indvariety_wpp} \& \eqref{eqn_2_wpp}, the ind-variety ${\mathbf{Q}}^{r,-kr}$ has the weak point property.  Hence we have the assertion.  
\end{proof}
\begin{theorem}\label{rank one case}
Let $C$ be a smooth projective curve over $\mathbb{C}$ and $n$ any given integer. Then the ind-variety of $(1,n)$-divisors on $C$ has the diagonal property.
\end{theorem}
\begin{proof}
Let $D$ be an effective divisor of degree $d$ on $C$.  Then we have, $Q^{1,-n}(D)=\Quot_{\mathcal{O}_C}^{d-n}$.  Now let's pick an effective divisor $D_1$ of degree $d_1$ satisfying the inequality $d_1>n$.  Then for all $D\geq D_1$, we have 
\begin{equation}\label{eqn_1_dp}
\deg(D)\geq \deg(D_1)=d_1>n,
\end{equation}
and  
\begin{equation}\label{eqn_2_dp}
Q^{1,-n}(D)=\Quot_{\mathcal{O}_C}^{\deg(D)-n}\cong \Sym^{\deg(D)-n}(C).
\end{equation}
Here $\deg(D)-n$ is a positive integer by \eqref{eqn_1_dp}.  Therefore, by Theorem \ref{symm prod_dp} and Definition \ref{indvariety_dp} \& \eqref{eqn_2_dp}, the ind-variety ${\mathbf{Q}}^{1,-n}$ of all $(1,n)$-divisors has the diagonal property.  
\end{proof}
\begin{remark}
It can be noted a particular case of Theorem \ref{integral slope case}, namely the case $r=1$, follows from Theorem \ref{rank one case} and Remark \ref{dp implies wpp}.
\end{remark}

\section{The diagonal property of the Hilbert scheme associated to a constant polynomial and its good partition}\label{sec: Hilbert scheme}
In this section, we talk about the Hilbert schemes associated to a polynomial and its good partitions.     First we mention the importance of studying such Hilbert schemes and then show that few of these Hilbert schemes satisfy the diagonal property.  Moreover, we provide an upper bound on the number of such Hilbert schemes.

Let $P(t)$ be a polynomial with rational coefficients.  We use the notation $\Hilb^{P}_C$ to denote the Hilbert scheme parametrizing all subschemes of $C$ having Hilbert polynomial $P(t)$.  Let $n$ be a positive integer.  Then interpreting $n$ as a constant polynomial, by $\Hilb^n_C$ we mean the Hilbert scheme parametrizing subschemes of $C$ having Hilbert polynomial $n$.  Let us recall the notion of a good partition of a polynomial and a Hilbert scheme associated to that. 
\begin{definition}\label{gp_defn}
	Let $\underline{P}=(P_i)_{i=1}^s$ be a family of polynomials with rational coefficients.  Then $\underline{P}$ is said to be a good partition of $P$ if $\sum_{i=1}^s P_i=P$ and $\Hilb_C^{P_i}\neq \phi$ for all $i$.
\end{definition} 
\begin{definition}
	The Hilbert scheme associated to a polynomial $P$ and its good partition $\underline{P}$, denoted by $\Hilb_C^{\underline{P}}$ , is defined as $\Hilb^{\underline{P}}_C:=\Hilb^{P_1}_C\times_\mathbb{C}\cdots \times_\mathbb{C} \Hilb^{P_s}_C$.
\end{definition}
\begin{remark}\label{motivation for checking dp for Hilbert schemes}
At this point it is worthwhile to mention the importance of the Hilbert scheme $\Hilb^{\underline{P}}_C$.  Recall that by  $\Quot^{P}_{\mathcal{F}}$ we denote the Quot scheme parametrizing all torsion quotients of $\mathcal{F}$ having having Hilbert polynomial $P(t)$.  We have a decomposition of $\Quot_{\mathcal{O}_C^r}^P$ as follows, whenever $\Quot_{\mathcal{O}_C^r}^P$ is smooth.
	\begin{equation*}\label{E8}
	\Quot_{\mathcal{O}_C^r}^P=\bigsqcup_{\substack{\underline{P}\; such\; that \;\underline{P}\\is \;a\; good\; partition\; of\; P }}\mathcal{S}_{\underline{P}},
	\end{equation*}
	where each $\mathcal{S}_{\underline{P}}$ is smooth, the torus $\mathbb{G}_m^r$-invariant, locally closed and isomorphic to a vector bundle over the scheme $\Hilb^{\underline{P}}_C$, (cf. \cite[p. 610]{B}).  Therefore, the cohomology of $\Quot_{\mathcal{O}_C^r}^P$ can be given by the direct sum of the cohomologies of $\Hilb^{\underline{P}}_C$, where the sum varies over the good partitions of the polynomial $P$.  So to study the cohomology ring $H^{\ast}(\Quot_{\mathcal{O}_C^r}^P)$, it is enough the cohomology rings $H^{\ast}(\Hilb^{\underline{P}}_C)$, $\underline{P}$ being good partition of the polynomial $P$.  Now, to get hold of the cohomology rings $H^{\ast}(\Hilb^{\underline{P}}_C)$, it's nice to get hold of the structure of the Hilbert scheme $\Hilb^{\underline{P}}_C$.  Now, as the diagonal property and the weak point property force strong conditions on the underlying variety (cf. \cite{PSP}), therefore to the study the cohomology of $\Quot_{\mathcal{O}_C^r}^P$ it's reasonable enough to check whether the Hilbert schemes $\Hilb^{\underline{P}}_C$'s posses these properties or not. 
\end{remark}
Remark \ref{motivation for checking dp for Hilbert schemes} motivates us to talk about the diagonal property of the Hilbert schemes associated to a constant polynomial and some particular good partitions of the same.  Towards that, we have the following Lemma.  
\begin{lemma}\label{partition_a good partition}
Let $n$ be a given positive integer.  Then any partition of $n$ is also a good partition of $n$ and vice versa.
\end{lemma}
\begin{proof}
Let $n$ be a positive integer.  An arbitrary partition of $n$ of length $s$ is given by a $s$-tuple $(n_1,n_2,\ldots,n_s)$ such that $\sum_{i=1}^r n_i=n$ and $n_i>0$ for all $i$.  As $n_i>0$ and $\Hilb^{n_i}_C$ is isomorphic to the moduli space $\Sym^{n_i}(C)$ of effective divisors of degree $n_i$ over $C$, we have $\Hilb^{n_i}_C\neq \emptyset$ for all $i$.  Therefore, by Definition \ref{gp_defn}, the chosen partition of $n$ is a good partition as well.   Converse part follows from the fact that given any given integer $k$, $\Hilb^{k}_C$ is non-empty if and only if $k$ is positive.  
\end{proof} 
 
We now have the following lemma which says that if two varieties have diagonal property, then so does their product.  The statement of the lemma can be found in the literature (cf. \cite[p. 1235]{PSP}, \cite[p. 47]{D}).  The proof though is not available to the best of our knowledge.  Therefore, for the sake of completeness, we briefly sketch the proof of the same.   
\begin{lemma}\label{dp for product}
Let $X_1$ and $X_2$ be two varieties over $\mathbb{C}$ satisfying the the diagonal property.  Then the product variety $X_1\times X_2$ also have the diagonal property. 
\end{lemma} 
\begin{proof}
Let $i=1,2$.  Let the dimension of $X_i$ be $n_i$.  As $X_i$ satisfy the diagonal property, by Definition \ref{dp}, there exists a vector bundle $E_i$ over $X_i\times X_i$ of rank $n_i$ and a section $s_i$ of $E_i$ such that the zero scheme $Z(s_i)$ of $s_i$ is the diagonal $\Delta_{X_i}$.  Let $p_i:(X_1\times X_2)\times (X_1\times X_2) \rightarrow X_i\times X_i$ are the projection maps given by $p_i((x_1,x_2,x'_1,x'_2))=(x_i,x'_i)$.  Consider the vector bundle $p_1^{\ast}E_1\oplus p_2^{\ast}E_2$ of rank $n_1+n_2$ over $(X_1\times X_2)\times (X_1\times X_2)$.  Then the zero scheme $Z((p_1^{\ast}s_1, p_2^{\ast}s_2))$ of the section $(p_1^{\ast}s_1, p_2^{\ast}s_2)$ of $p_1^{\ast}E_1\oplus p_2^{\ast}E_2$ is the diagonal $\Delta_{X_1\times X_2}$ of $X_1\times X_2$.  Hence, the assertion follows by Definition \ref{dp}.
\end{proof} 

By multiprojective space, we mean product of projective spaces.  The following lemma says that given two distinct partition of a positive integer $n$, the corresponding multiprojective spaces are not isomorphic.  More precisely, we have the following.
\begin{lemma}\label{non isomorphic multi projective spaces}
Let $n$ be a positive integer.  Let $(m_1,m_2,\ldots,m_s)$ and $(n_1,n_2,\ldots,n_s)$ be two distinct partitions of $n$ of same length $s$.  Then $\mathbb{P}^{m_1}\times 
\mathbb{P}^{m_2}\times \cdots \times \mathbb{P}^{m_s}$ is not isomorphic to  $\mathbb{P}^{n_1}\times \mathbb{P}^{n_2}\times \cdots \times \mathbb{P}^{n_s}$.
\end{lemma}
\begin{proof}
The cohomology ring $H^{\ast}(\mathbb{P}^{m_i},\mathbb{Z})$ of $\mathbb{P}^{m_i}$ is the ring $\tfrac{\mathbb{Z}[x_i]}{\langle x_i^{m_i+1} \rangle}$, where $x_i$ is a generator of $H^2(\mathbb{P}^{m_i},\mathbb{Z})(\simeq \mathbb{Z})$, for all $1\leq i\leq s$.  Similarly,  $H^{\ast}(\mathbb{P}^{n_i},\mathbb{Z})=\tfrac{\mathbb{Z}[y_i]}{\langle y_i^{n_i+1} \rangle}$, for all $1\leq i\leq s$.  Now, by K\"unneth formula, we have 
\begin{equation*}\label{cohomology of multiprojective space}
\begin{split}
H^{\ast}(\mathbb{P}^{m_1}\times 
\mathbb{P}^{m_2}\times \cdots \times \mathbb{P}^{m_s},&\mathbb{Z})=\dfrac{\mathbb{Z}[x_1]}{\langle x_1^{m_1+1} \rangle}\otimes \dfrac{\mathbb{Z}[x_2]}{\langle x_2^{m_2+1} \rangle}\otimes \cdots \otimes \dfrac{\mathbb{Z}[x_s]}{\langle x_s^{m_s+1} \rangle}\\
&=\dfrac{\mathbb{Z}[x_1,x_2,\cdots, x_s]}{\langle x_1^{m_1+1}, x_2^{m_2+1}, \cdots x_s^{m_s+1} \rangle}=M\,(\text{say}).
\end{split}
\end{equation*} 
Similarly, we have $H^{\ast}(\mathbb{P}^{n_1}\times 
\mathbb{P}^{n_2}\times \cdots \times \mathbb{P}^{n_s},\mathbb{Z})=\tfrac{\mathbb{Z}[y_1,y_2,\cdots, y_s]}{\langle y_1^{n_1+1}, y_2^{n_2+1}, \cdots y_s^{n_s+1} \rangle}=N\,(\text{say})$.
Now, for all $1\leq i\leq s$, let 
\begin{equation*}
pr_{m_i}: \mathbb{P}^{m_1}\times 
\mathbb{P}^{m_2}\times \cdots \times \mathbb{P}^{m_s} \longrightarrow \mathbb{P}^{m_i}
\end{equation*}
be the $i$-th projection map.  Then, in the ring $M$, $x_i$ can be interpreted as the first Chern class $c_1(pr_{m_i}^{\ast}(\mathcal{O}_{\mathbb{P}^{m_i}}(1)))$ of the pullback of the hyperplane bundle $\mathcal{O}_{\mathbb{P}^{m_i}}(1)$ on $\mathbb{P}^{m_i}$ via the projection map $pr_{m_i}$.  Moreover, under the identification of the Picard group $\Pic(\mathbb{P}^{m_1}\times 
\mathbb{P}^{m_2}\times \cdots \times \mathbb{P}^{m_s})$ of $\mathbb{P}^{m_1}\times 
\mathbb{P}^{m_2}\times \cdots \times \mathbb{P}^{m_s}$ with direct sum of $s$ copies of $\mathbb{Z}$, the line bundle $pr_{m_i}^{\ast}(\mathcal{O}_{\mathbb{P}^{m_i}}(1))$ is nothing but $(0,0,\ldots,1,\dots,0)$, $1$ being in the $i$-th place.  Therefore, $pr_{m_i}^{\ast}(\mathcal{O}_{\mathbb{P}^{m_i}}(1))$ is globally generated (cf. \cite[Theorem 7.1, p. 150]{Ha}).  Moreover, as $\mathbb{P}^{m_1}\times \mathbb{P}^{m_2}\times \cdots \times \mathbb{P}^{m_s}$, being a projective variety, is complete and therefore $pr_{m_i}^{\ast}(\mathcal{O}_{\mathbb{P}^{m_i}}(1))$ is nef (cf. \cite[Example 1.4.5, p. 51]{La}), but not ample (cf. \cite[Example 7.6.2, p. 156]{Ha}), for all $1\leq i\leq s$.  Following similar notations, $y_i=c_1(pr_{n_i}^{\ast}(\mathcal{O}_{\mathbb{P}^{n_i}}(1)))$, where the bundle $pr_{n_i}^{\ast}(\mathcal{O}_{\mathbb{P}^{n_i}}(1))$ is nef, but not ample.  So, if there exists an isomorphism between $\mathbb{P}^{m_1}\times \mathbb{P}^{m_2}\times \cdots \times \mathbb{P}^{m_s}$ and $\mathbb{P}^{n_1}\times \mathbb{P}^{n_2}\times \cdots \times \mathbb{P}^{n_s}$, then $pr_{m_i}^{\ast}(\mathcal{O}_{\mathbb{P}^{m_i}}(1))$ must go to some $pr_{n_j}^{\ast}(\mathcal{O}_{\mathbb{P}^{n_j}}(1))$ under that isomorphism (cf. \cite[Example 1.4.5, p. 51]{La}).  Therefore, $x_i$ should map to some $y_j$ under the induced isomorphism at the cohomology level, that is to say $(x_1,x_2, \cdots, x_s)=(y_{\sigma(1)},y_{\sigma(2)}, \cdots, y_{\sigma(s)})$, for some $\sigma \in S_s$.  Now, as the partitions $(m_1,m_2,\ldots,m_s)$ and $(n_1,n_2,\ldots,n_s)$ of $n$ are distinct, $m_i\neq n_i$ for some $i$.  So, the rings $M$ and $N$ are not isomorphic.  But that is a contradiction to our assumption that $\mathbb{P}^{m_1}\times 
\mathbb{P}^{m_2}\times \cdots \times \mathbb{P}^{m_s}$ and $\mathbb{P}^{n_1}\times \mathbb{P}^{n_2}\times \cdots \times \mathbb{P}^{n_s}$ are isomorphic.  Hence, the assertion follows.
\end{proof}

Finally, we have the following theorem which says about the diagonal property of the Hilbert schemes associated to a constant polynomial and its good partitions.  Moreover, we provide an upper bound on the number of such Hilbert schemes and prove that the obtained bound is the best possible bound.

\begin{theorem}
Let $C$ be a smooth projective curve over $\mathbb{C}$ and $n$ a positive integer.  Then there are at most as many Hilbert schemes $\Hilb^{\underline{n}}_C$ (upto isomorphism) associated to the constant polynomial $n$ and its good partitions $\underline{n}$ satisfying diagonal property as there are conjugacy classes of the symmetric group $S_n$ of $n$ symbols.  Moreover, this upper bound is achieved when $C$ is of genus $0$ and hence is sharp.
\end{theorem}
\begin{proof}
For a positive integer $n$, a partition of $n$ of length $s$ is given by a $r$-tuple $(n_1,n_2,\ldots,n_s)$ such that $\sum_{i=1}^s n_i=n$ and $n_1\geq n_2\geq \cdots \geq n_s>0$. Then $(n_1,n_2,\ldots,n_s)$ is also a good partition of $n$ by Lemma \ref{partition_a good partition}. Moreover, the associated Hilbert scheme is given by $\Hilb^{n_1}_C\times_\mathbb{C} \Hilb^{n_2}_C\times_\mathbb{C} \cdots \times_\mathbb{C} \Hilb^{n_s}_C$.  As $\Hilb^m_C\cong \Sym^m(C)$ for any positive integer $m$, by Theorem \ref{symm prod_dp} and Lemma \ref{dp for product}, we get that the associated Hilbert scheme $\Hilb^{n_1}_C\times_\mathbb{C} \Hilb^{n_2}_C\times_\mathbb{C} \cdots \times_\mathbb{C} \Hilb^{n_s}_C$ satisfies the diagonal property.  Therefore, given any arbitrary partition of $n$, the associated Hilbert scheme has the diagonal property.  

Now let us take two distinct partition of $n$, say $(n_1,n_2,\ldots,n_s)$ and $(n^{'}_1,n^{'}_2,\ldots,n^{'}_t)$.  Then we have the following two mutually exclusive and exhaustive cases:\\
\textit{First Case :} $s\neq t$\\
In this case, the associated Hilbert schemes $\Hilb^{n_1}_C\times_\mathbb{C} \Hilb^{n_2}_C\times_\mathbb{C} \cdots \times_\mathbb{C} \Hilb^{n_s}_C$ and $\Hilb^{n^{'}_1}_C\times_\mathbb{C} \Hilb^{n^{'}_2}_C\times_\mathbb{C} \cdots \times_\mathbb{C} \Hilb^{n^{'}_t}_C$ are distinct as they can be written as product of different number of $\Hilb^m_C$'s.\\ 
\textit{Second Case :} $s=t$\\
In this case, as $n_i\neq n^{'}_i$ for some $1\leq i \leq s$, therefore $\Hilb^{n_i}_C\neq \Hilb^{n^{'}_i}_C$.  Hence the associated Hilbert schemes $\Hilb^{n_1}_C\times_\mathbb{C} \Hilb^{n_2}_C\times_\mathbb{C} \cdots \times_\mathbb{C} \Hilb^{n_s}_C$ and $\Hilb^{n^{'}_1}_C\times_\mathbb{C} \Hilb^{n^{'}_2}_C\times_\mathbb{C} \cdots \times_\mathbb{C} \Hilb^{n^{'}_s}_C$ are not same as well.
  
Therefore, we conclude that any two distinct partitions of $n$ give us two distinct Hilbert schemes associated to those partitions satisfying the diagonal property.   So, upto isomorphism, there could be at most as many such Hilbert schemes as there are partitions of $n$.  Now as number of conjugacy classes of $S_n$ is equal to the number of partition $p(n)$ of $n$ (cf. \cite[Lemma 2.11.3, p. 89]{H}), the first part of the assertion follows from Lemma \ref{partition_a good partition}.

We now show that the obtained upper bound for number of Hilbert schemes associated to the good partitions of the constant polynomial $n$ satisfying diagonal property is in fact is achieved in genus $0$ case.  Indeed, let us consider $C=\mathbb{P}^1$.  Then, $\Hilb^n_{\mathbb{P}^1}=\Sym^n(\mathbb{P}^1)=\mathbb{P}^n$.  Now let us take two distinct partition of $n$, say $(n_1,n_2,\ldots,n_s)$ and $(n^{'}_1,n^{'}_2,\ldots,n^{'}_t)$.  Then, as before, we have the following two mutually exclusive and exhaustive cases:\\
\textit{First Case :} $s\neq t$\\
In this case, the associated Hilbert schemes $\Hilb^{n_1}_{\mathbb{P}^1}\times_\mathbb{C} \Hilb^{n_2}_{\mathbb{P}^1}\times_\mathbb{C} \cdots \times_\mathbb{C} \Hilb^{n_s}_{\mathbb{P}^1}$ and $\Hilb^{n^{'}_1}_{\mathbb{P}^1}\times_\mathbb{C} \Hilb^{n^{'}_2}_{\mathbb{P}^1}\times_\mathbb{C} \cdots \times_\mathbb{C} \Hilb^{n^{'}_t}_{\mathbb{P}^1}$ are not isomorphic as their Picard groups are not so.  That is,
\begin{equation*}
\begin{split}
\Pic(\Hilb^{n_1}_{\mathbb{P}^1}\times_\mathbb{C} \Hilb^{n_2}_{\mathbb{P}^1}\times_\mathbb{C} \cdots \times_\mathbb{C} \Hilb^{n_s}_{\mathbb{P}^1})&\cong \oplus_{i=1}^s\mathbb{Z}\ncong \oplus_{i=1}^t\mathbb{Z}\\
&=\Pic(\Hilb^{n^{'}_1}_{\mathbb{P}^1}\times_\mathbb{C} \Hilb^{n^{'}_2}_{\mathbb{P}^1}\times_\mathbb{C} \cdots \times_\mathbb{C} \Hilb^{n^{'}_t}_{\mathbb{P}^1}).
\end{split}
\end{equation*}
\textit{Second Case :} $s=t$ \\
In this case, the associated Hilbert schemes $\Hilb^{n_1}_{\mathbb{P}^1}\times_\mathbb{C} \Hilb^{n_2}_{\mathbb{P}^1}\times_\mathbb{C} \cdots \times_\mathbb{C} \Hilb^{n_s}_{\mathbb{P}^1}$ and $\Hilb^{n^{'}_1}_{\mathbb{P}^1}\times_\mathbb{C} \Hilb^{n^{'}_2}_{\mathbb{P}^1}\times_\mathbb{C} \cdots \times_\mathbb{C} \Hilb^{n^{'}_s}_{\mathbb{P}^1}$ are not isomorphic by Lemma \ref{non isomorphic multi projective spaces}.   

Hence the second part of the assertion follows. 
\end{proof}

\section*{Statements and Declarations}
The author declare that there is no potential competing interest.

\section*{Acknowledgements}
The author would like to thank Prof. D. S. Nagaraj and Dr. C. Gangopadhyay for many useful discussions and throughout encouragements.  The author also wishes to thank Indian Institute of Science Education and Research Tirupati for financial support (Award No.-IISER-T/Offer/PDRF/A.M./M/01/2021).


\end{document}